\documentclass[microtype]{gtpart}
\usepackage[all]{xy}
\usepackage[parfill]{parskip}
\usepackage{amsmath}

\usepackage{amscd, latexsym, hyperref, times}
\usepackage{color,graphicx, url}
\newtheorem{dummy}{anything}[section]
\newtheorem{theorem}[dummy]{Theorem}

\newtheorem{lemma}[dummy]{Lemma}

\newtheorem{proposition}[dummy]{Proposition}

\newtheorem{corollary}[dummy]{Corollary}

\theoremstyle{definition}
\newtheorem{definition}[dummy]{Definition}

\newtheorem{remark}[dummy]{Remark}
\newtheorem*{rem}{Remark}

\renewcommand
{\theequation}{\thesection.\arabic{dummy}}

\newcommand{\bR}{\mathbb R}

\newcommand{\bfc}{{\mathbb{C}}}
\newcommand{\Crit}{\rm{Crit\/}}

\newcommand{\La}{\Lambda}

\newcommand{\bfd}{{\mathbb{D}}}

\newcommand{\bfq}{{\mathbb{Q}}}

\newcommand{\bfz}{{\mathbb{Z}}}
\newcommand{\f}[1]{\mathbb{#1}}
\newcommand{\set}[1]{\{ #1 \}}

\newcommand{\QED}{\vspace{-.32in}\begin{flushright}\qed\end{flushright}}
 
 \newcommand{\del}{\partial} 
\begin{document}

\title{The symplectic topology of some rational homology balls}

\author{Yank\i\ Lekili} 
\email{yl319@cam.ac.uk}
\address{King's College, University of Cambridge}
\author{Maksim Maydanskiy}
\email{maksimm@math.stanford.edu}
\address{Stanford University}

\begin{abstract}
We study the symplectic topology of some finite algebraic quotients of the $A_n$ Milnor fibre which are diffeomorphic to the rational homology balls that appear in Fintushel and Stern's rational blowdown
construction. We prove that these affine surfaces have no closed exact Lagrangian submanifolds by using the already available and deep understanding of the
Fukaya category of the $A_n$ Milnor fibre coming from homological mirror
symmetry. On the other hand, we find Floer theoretically essential monotone
Lagrangian tori, finitely covered by the monotone tori which we study in the
$A_n$ Milnor fibre. We conclude that these affine surfaces have non-vanishing
symplectic cohomology.  \end{abstract}

\maketitle

\section{Introduction}

Let $p>q>0$ be relatively prime integers. In \cite{CH}, Casson and Harer
introduced rational homology balls $B_{p,q}$ which are bounded by the lens
space $L(p^2,pq-1)$. These homology balls were subsequently used in
Fintushel-Stern's rational blow-down construction \cite{FS} (see also,
\cite{JP}).  In fact, $B_{p,q}$ are naturally equipped with Stein structures
since they are affine varieties (cf. \cite{KSB}) and here we are concerned with the symplectic
topology of these Stein surfaces.

The key topological fact is that $B_{p,q}$ are $p$-fold covered (without
ramification) by the Milnor fibre of the $A_{p-1}$ singularity. The latter has
a unique Stein structure and its symplectic topology is well-studied (see
\cite{SK}, \cite{SM}, \cite{SS}, \cite{book}).

Following Seidel \cite{SeidelEvans}, we make the following definition:

\begin{definition} A Stein manifold $X$ is said to be empty if its symplectic
cohomology vanishes. It is non-empty otherwise.  \end{definition}

We recommend \cite{biased} for an excellent survey of symplectic cohomology. 
Non-empty Stein manifolds are often detected by the following important theorem
of Viterbo (here stated in a weak form):

\begin{theorem}[Viterbo, \cite{Viterbo1}]\label{vit} Let $X$ be a Stein
manifold. If there exists a closed exact Lagrangian submanifold of $X$ then $X$
is non-empty.  \end{theorem}

The question of existence of closed exact Lagrangian submanifolds goes back to Gromov, who proved that no such submanifold exists in $\f{C}^{n}$ (see \cite[Corollary 2.3.$B_2$]{Gromov}). Of course $\f{C}^{n}$ also has vanishing symplectic cohomology (as explained, for example, in \cite[Section 3f]{biased}), which together with the above theorem reproves Gromov's result.

We will exploit the relation of $B_{p,q}$ with the Milnor fibre of $A_{p-1}$
singularity to prove the following theorem:

\begin{theorem} For $p\neq 2$, the affine surface $B_{p,q}$ has no closed exact
Lagrangian submanifolds. On the other hand, $B_{p,q}$ contains a Floer
theoretically essential Lagrangian torus, therefore $B_{p,q}$ is non-empty.
\end{theorem}

Although one expects to find many non-empty Stein surfaces with no closed exact
Lagrangian submanifolds, to our knowledge, the above examples represent the
first explicit construction of non-empty Stein surfaces with no closed exact
Lagrangian submanifolds.  In higher dimensions (dim$_\f{R} \geq 12$), Abouzaid and Seidel exhibited infinitely many examples in
\cite{recomb} where symplectic cohomology is non-zero with coefficients in
$\f{Z}$ but vanishes with coefficients in $\f{Z}_2$. Such examples obviously
cannot contain closed exact Lagrangian submanifolds by Viterbo's theorem
applied over $\f{Z}_2$. Our examples are not only of lower dimension but also
have non-vanishing symplectic cohomology with arbitrary coefficients.
Therefore, there is no direct way of appealing to Viterbo's theorem in order to
exclude existence of closed exact Lagrangian submanifolds. Their non-emptiness
is detected by the existence of Floer theoretically essential tori
(\cite{biased} Proposition 5.2). On the other hand, the non-existence of closed
exact Lagrangians is proved using a detailed understanding of closed exact
Lagrangians in the $A_n$ Milnor fibres based on twisted symplectic cohomology
applied by Ritter in \cite{ritter} which suffices for $p$ odd. For $p$ even, we
utilize a deeper understanding coming from homological mirror symmetry and
calculations on the $B$-model provided by Ishii, Ueda, Uehara \cite{iu},
\cite{iuu}. It is remarkable that algebro-geometric calculations on the mirror
side can be utilized profitably towards an application to symplectic topology.  

En route, we study a class of tori in $A_n$ Milnor fibres, which we call
matching tori (cf. matching spheres \cite{AMP} ). We will classify them up to
Hamiltonian isotopy and show that the Floer cohomology of these tori is
non-zero. This fact is probably known to experts in the field; however as we
did not find a written account of this result for $n \geq 2$, we take this
opportunity to provide a proof as this fact will be used in proving our main
result above.

{\bf Acknowledgments:} YL was supported by Herchel Smith Fund and Marie Curie grant EU-FP7-268389. MM was supported by NSF grant DMS-0902763 and ERC grant ERC-2007-StG-205349.

\section{Lagrangian tori in $A_n$ Milnor fibres and their Floer cohomology}
\label{section2}

\subsection{$A_n$ Milnor fibre}

The four-dimensional $A_n$ Milnor fibre is given by the affine hypersurface\footnote{The factor 2 is for compatibility with the conventional description of $A_n$ Milnor fibre given by $\{ (x,y,z) \in \f{C}^3 : z^{n+1} +x^2+y^2=1 \}$ } : \[
S_n = \{ (x,y,z) \in \f{C}^3 : z^{n+1}+2xy=1 \} .\] $S_n$ has the induced
complex structure as a subvariety of $\f{C}^3$, which makes it into a Stein
manifold, and can be equipped with the exact symplectic form inherited from the
standard form on $\f{C}^3$ given by: \[ \omega= d\theta = \frac{i}{2} \left( dx
\wedge d\bar{x} + dy \wedge d\bar{y} + dz \wedge d\bar{z} \right) \] where
$\theta = \frac{i}{4} \left( x d\bar{x} - \bar{x} dx  + y d\bar{y} - \bar{y} dy
+ z d\bar{z} - \bar{z} dz \right). $  

Due to the existence of many exact Lagrangian spheres in $S_n$, this
hypersurface has been instrumental in constructing many interesting examples in
symplectic geometry (see \cite{SK}, \cite{SM}, \cite{SS}, \cite{book}). We will
recall some generalities about $S_n$, and we refer the reader to loc. cit.
for more. 

The projection $\Pi_n \colon S_n \to \f{C}$ to the $z$-coordinate yields an
exact Lefschetz fibration with $n+1$ critical points at the roots of unity,
which is adapted to the Stein structure. The fibre of this Lefschetz fibration
is a one-sheeted hyperboloid. The vanishing cycle of any critical point and
any vanishing path in the regular fibre is always the core of the hyperboloid,
given by \[ V_x = \Pi^{-1}(z) \cap \{ (x,y,z) \in \f{C}^3 : |x| = |y| \} .\]

Let $D_{r}$ be the disk of radius $r$ centered at the origin in the base, and
$C_{r} = \del D_{r}$. For $r > 1$, the three-manifolds $Y_r =
\del(\Pi^{-1}(D_r))$ are all diffeomorphic to the lens space $L(n+1,n)$ and are
equipped with the unique tight contact structure on $L(n+1,n)$ induced by the
filling provided by $\Pi^{-1}(D_r)$. (The fact that there is a unique tight
contact structure on $L(n+1,n)$ is used below and follows from the
classification of tight contact structures on lens spaces, see \cite{Giroux},
\cite{Honda}). The restriction $\Pi|_{Y_r}$ provides an open book supporting
this contact structure and its monodromy is given by $(n+1)^{th}$ power of the
right-handed Dehn twist along the core of the fibre. Note that the fibre has
genus zero and clearly there is a unique factorization of this monodoromy into
a product of right-handed Dehn twists, therefore Wendl's theorem (\cite{wendl})
implies that there is a unique Stein structure on $S_n$ up to deformation,
namely the one coming from the restriction of the standard structure on
$\f{C}^{3}$.  In this way, we obtain an exact symplectic manifold $(S_n,
d\theta)$ with $c_1(S_n) =0$. Topologically, $S_n$ is a linear plumbing of $n$
disk bundles over $S^2$ with Euler number $-2$.

Next, we consider a  family of Lagrangian tori in $A_n$ Milnor fibre.  We call
the tori in this family matching tori, as they are obtained as unions of
vanishing cycles. 

\subsection{Matching tori}

Given a Lefschetz fibration $\Pi: E \mapsto \bfd^2$,  a closed embedded circle $\gamma:
[0,1] \mapsto \bfd^2$ with $\gamma(0)=\gamma(1)=p$ and a Lagrangian $V$ in the
fibre $F_p=\Pi^{-1}(p)$, such that the symplectic monodromy $\phi$ along $\gamma$
takes $V$ to itself, we define a matching Lagrangian $L$  to be the union of
all parallel translates of $V$ over $\gamma$. Explicitly $L = \bigcup_{x \in
\gamma} V_x$ where $V_x$ is the parallel transport of $V$ over $\gamma$ to
$\Pi^{-1}(x)$. 

Note that $L$ is diffeomorphic to the mapping torus of $\phi|_V$ and is in fact
a Lagrangian submanifold of $E$ by Lemma 16.3 in \cite{book}.  In the case when
dimension of $E$ is 4, and  $V$ is a circle, we call $L$ the \emph{matching
torus of $V$ along $\gamma$}.  
  
In the case of the Milnor fibre $S_n$, we take the closed path $\gamma$
oriented in such a way that the enclosed area is positive.  We call the
resulting Lagrangian torus $\f{T}_{n;\gamma}$ or $\f{T}_n$ if the particular
choice of $\gamma$ is not important. 

We will see below that the tori $\f{T}_n$ bound holomorphic disks, in
particular they cannot be exact Lagrangian submanifolds of $S_n$. In contrast,
there is an abundance of exact Lagrangian spheres obtained by matching sphere
construction, which we recall now. Take an embedded path $c : [0,1] \mapsto
\bfd^2$ such that $c^{-1}(\text{Critv}(\Pi))= \{0,1\}$. To such a path $c$ one can
associate an exact Lagrangian sphere $V_c$ defined explicitly as the union of
vanishing cycles over $c$: $V_c = \bigcup_{x \in c} V_x $ where $V_x =
\Pi^{-1}(z) \cap \{ (x,y,z) \in \f{C}^3 : |x| = |y| \}$. The fact that $V_c$ is
an exact Lagrangian can be seen by observing that it can be split up as a union
of Lefschetz thimbles for the Lefschetz fibration $\Pi_n$ (see \cite{book}
16g). We note that the core spheres in the plumbing picture of $S_n$ can be
taken to be matching spheres of linear paths connecting the critical values. 

As $S_n$ is simply connected and $\pi_2(\f{T}_n)=0$, from the homotopy exact
sequence we have: \[ 0 \to \pi_2(S_n) \to \pi_2 (S_n, \f{T}_n) \to
\pi_1(\f{T}_n) \to 0 \] which splits as $\pi_1(\f{T}_n)=\bfz^2$ is free.
$\pi_2(S_n)$ is generated by the cores of the disk bundles in the plumbing
description of $S_n$ and are represented by Lagrangian matching spheres, hence
they have zero Maslov index and symplectic area. Furthermore, one of the $\bfz$
factors in $\pi_1(\f{T}_n)$ is generated by the vanishing cycle $V$, which is
the boundary of a Lefschetz thimble. Since the thimble is a Lagrangian $\mathbb{D}^2$, again its Maslov
index and the symplectic area vanishes. It remains to determine the index and
the area on a class $\beta \in \pi_2(S_n, \f{T}_n)$ such that $\Pi$ restricted
to $\partial\beta$ is a degree $1$ map onto $\gamma$. For this purpose, we will
need a more explicit computation. 

Let us consider the parametrized curve $c(t) = ( n(t) e^{2\pi i
(\alpha(t)+\beta(t))} , n(t) e^{2\pi i(\alpha(t)-\beta(t))}, \gamma(t) ) $ for
$t \in [0,1]$, $n(t)>0 , \alpha(t),\beta(t)$ real valued functions and $\gamma(t)$ is a degree 1 parametrization of $\gamma$ such that $2 n(t)^2 e^{4\pi i\alpha(t)} = 1- \gamma(t)^{n+1}$.
Then $c(t)$ is a curve on $\f{T}_n$ mapping onto $\gamma$ with a degree 1 map.
The area of any disk with boundary on such a curve is a sum of the areas of its
three coordinate projections. This area is given by the integral of $\theta$
over the curve by Stokes' theorem. We compute: \[ \int_{c} \theta =
\int_{\gamma} \frac{i}{4} ( z d\bar{z} - \bar{z} dz )   + 2\pi \int_{0}^1 \alpha'(t)
n(t)^2 dt. \] Here the first term is the area enclosed by  the projection of
$c(t)$ on the $z$ coordinate plane, and the second term is the sum of the two
other area contributions. Note that the integral is independent of $\beta(t)$.
This is a reflection of the fact the integral of $\theta$ over $V$ is zero, hence we could have taken a curve $c(t)$ with $\beta(t)=0$. 
	
\begin{lemma}\label{integrand}	$I= 2 \pi \int_{0}^1 \alpha'(t) n(t)^2 dt > 0$.
\end{lemma}	

\begin{proof} 

Consider the map $f: \bfc \setminus \text{Critv}(\Pi_n) \mapsto \bfc\setminus \{0\}$,  $f(z)=\frac{1-z^{n+1}}{\sqrt{2|1-z^{n+1}|}}$.
 This is a composition of the holomorphic map $p(z)= 1-z^{n+1}$ and a smooth orientation preserving map $s(r e^{i\theta})= \sqrt{\frac{r}{2}} e^{i\theta}$ and has a continuous extension to $F: \bfc \mapsto \bfc$.  
The map $f$  sends $\gamma(t)$ to $n(t) e^{4\pi i\alpha(t)}$, so the integral we are interested in is $I =  \int_{f(\gamma)} \frac{1}{2} r^2 d\theta$. 
  If we denote the interior of $\gamma$ (in $\bfc$) by $G$, then $F(G)$ has boundary $f(\gamma)$ and by Stokes theorem, $I=\int_{F(G)} r dr \wedge d \theta=\int_{f(G\setminus \text{Critv} (\Pi_n))}  r dr \wedge d \theta$, which is positive since $f$ is orientation preserving.  \end{proof}
		
We note in addition, that $I= \int_{f(G\setminus \text{Critv}(\Pi_n))}
	r dr \wedge d \theta = \int_{(G\setminus \text{Critv}(\Pi_n))}  f^*(r
	dr \wedge d \theta)$. We put  \[ \sigma= \sigma_0+ f^*(r dr \wedge d
	\theta) \] where $\sigma_0 = \frac{i}{2} dz \wedge d\bar{z}$ is the standard area form on $\bfc\setminus
	\text{Critv} (\Pi_n)$ induced from $\bfc$, and note that since $f$ is
	orientation preserving, $\sigma$ is an area form on  $\bfc\setminus
	\text{Critv} (\Pi_n)$. Observe that $\sigma$ blows up near
	$\text{Critv}(\Pi_n)$ but is integrable across them, so for a region
	$G$  in $\bfc$ one  always has a finite integral  $\int_G
	\sigma=\int_{G\setminus \text{Critv}(\Pi_n)} \sigma$. 

We summarize this discussion in the following lemma:

\begin{lemma} \label{areas} Let $\beta \in \pi_2(S_n, \f{T}_n)$, $\omega$ and $\sigma$ are symplectic forms on $S_n$ and $\f{C} \backslash{\text{Critv}(\Pi)}$ as above. Then we have: \[ \int_\beta \omega = \int_{\Pi(\beta)} \sigma \]
\end{lemma} 
\begin{proof} Note that for the classes in $\pi_2(S_n, \f{T}_n)$ that are represented by matching spheres or a Lagrangian thimble both integrals vanish. The equality for any other class follows from the computations above.
\end{proof} 

Note that varying the path $\gamma$ outside of the critical value set of $\Pi$
leads to a Lagrangian isotopy of $\f{T}_n$.  We remark that this isotopy is
Hamiltonian if and only if it is exact (see \cite[Section 6.1]{Polterovich}),
which in the situation at hand is equivalent to the symplectic area of the disk
discussed above staying constant during the isotopy (as the other generator
still bounds the thimble during such an isotopy). Therefore, we define  \[ \tau_\gamma
:= \int_{\gamma} \frac{i}{4} ( z d\bar{z} - \bar{z} dz)   + 2\pi  \int_{0}^1 \alpha'(t) n(t)^2 dt \] which will be called the
$\emph{monotonicity constant}$.     Note that  $\tau_\gamma$ is the area of region enclosed by $\gamma$ with respect to $\sigma$. 
  
 \begin{lemma} \label{curves} Let $(\Sigma, \sigma=d \lambda)$ be an exact
 symplectic 2-manifold.  Then any two isotopic closed curves $\gamma_0,
 \gamma_1$ with $\int_{\gamma_0} \lambda = \int_{\gamma_1} \lambda$   are
 Hamiltonian isotopic via a compactly supported Hamiltonian isotopy.
 \end{lemma} 
 
 \begin{proof} This is an adaptation of Proposition A.1 in \cite{salamon} and
 follows a similar route - we first extend the isotopy connecting
 $\gamma_0$ to $\gamma_1$ to a global smooth isotopy such that it ends at a
 symplectomorphism, then use parametrised Moser's trick to get a symplectic
 isotopy connecting $\gamma_0$ to $\gamma_1$ (this works only in dimension 2) .
 Finally, this gives an exact Lagrangian, and hence Hamiltonian, isotopy
 between $\gamma_0$ and $\gamma_1$. The proof below follows this outline.
 
Let $f_t$ be an isotopy connecting $\gamma_0$ to $\gamma_1$, that is a map $f
\colon S^1\times [0,1] \to \Sigma$ such that $f_t = f_{S^1 \times \set{t} }$ is
an embedding. Without loss of generality, it suffices to restrict to a compact
submanifold of $\Sigma$ which contains the image of $f$ and in which $\gamma_0$
(and hence $\gamma_1$) is separating (in fact, $f_t$ can be taken to be supported near
embedded annuli with core $\gamma_0$ and bigons between $\gamma_0$ and
$\gamma_1$; it is easy to see that support of such an isotopy is contained in a
sub-annulus of $\Sigma$ with core $\gamma_0$). 

We first prove that $f_t$ extends to symplectic isotopy $F_t$ of a
neighbourhood of $\gamma_0$ (cf. Ex. 3.40 in \cite{McDS}). Write $\gamma_t =f_t(\gamma_0)$ and choose an increasing
sequence $t_k \in [0,1]$ starting with $t_0=0$ and ending with $t_N=1$, such
that $\gamma_{t_{k+1}}$ is in a Weinstein tubular neighbourhood of
$\gamma_{t_k}$ and for $t \in [t_k,t_{k+1} ]$, $\gamma_t$ are graphs of closed 1-forms $\mu_t$ on
$\gamma_{t_k}$. Then $(q, p) \mapsto (q,
p+\mu_t(q))$ is a symplectic isotopy extending $f_t$ on a neighbourhood of
$\gamma_{t_k}$. Taking the neighbourhood of $\gamma_0$ small enough that it
lands in the domain of definition of these extensions for all times $t \in
[0,1]$ gives the desired extension $F_t$.
 
The isotopy $F_t$ constructed above extends to a smooth compactly supported
isotopy of $\Sigma$ which coincides with $F_t$ on a smaller neighbourhood of
$\nu(\gamma_0)$ of $\gamma_0$. We denote this isotopy of $\Sigma$ by $G_t$. We next show that by a
further compactly supported smooth isotopy we can replace $G_t$ with an isotopy
$H_t$ such that $G_t$ and $H_t$ agree near $\gamma_0$ and $H_1$ is a
symplectomorphism.

Consider the compactly supported closed 2-form $G_1^* \sigma-\sigma$. Since
$G_1$ is a symplectomorphism near $\gamma_0$, $G_1^*\sigma -\sigma$ vanishes
identically near $\gamma_0$. Hence, it represents a class in $H^2_c(\Sigma ,
\nu(\gamma_0) )$. In addition, as $\gamma_0$  is separating, this last group is rank
2 corresponding to two connected pieces of $\Sigma \backslash \nu(\gamma_0)$ , say
$\Sigma_1$ and $\Sigma_2$. 

Now, by Stokes' theorem and since $G_t$ is compactly supported, we get
$\int_{\Sigma_1} (G_1^*(\sigma)-\sigma) = \int_{\gamma_0}
(G_1^*(\lambda)-\lambda) = \int_{\gamma_1} \lambda-\int_{\gamma_0} \lambda =0$, and
similarly the integral of $G_1^*\sigma-\sigma$ vanishes over $\Sigma_2$.
Therefore, we have $G_1^*\sigma-\sigma =d\alpha$ for some compactly supported
$1$-form $\alpha$ which vanishes near $\gamma_0$.
  
Hence we have a  family of symplectic forms $\sigma_s = (1-s) \sigma + s
G_1^*(\sigma)= \sigma + s d \alpha$ and Moser's trick yields the desired
isotopy. In summary, we have produces a compactly supported isotopy $H_t$ such
that $H_1$ sends $\gamma_0$ to $\gamma_1$ and is a symplectomorphism on
$\Sigma$.

We have the following lemma: 

\begin{lemma} \label{paramoser} A compactly supported symplectomorphism $H_1$
smoothly isotopic to identity via a compactly supported isotopy is isotopic to
identity via a compactly supported family of symplectomorphisms.  \end{lemma}
\begin{proof} This is a compactly supported version of Lemma A2 in
\cite{salamon}, and is an application of a parametrised Moser's trick.
\end{proof}
 
\emph{Completion of the proof of Lemma \ref{curves}}:  The Lemma \ref{paramoser}
applied to $H_t$ yields a symplectic isotopy $K_t$ connecting $\gamma_0$ to
$\gamma_1$. Finally, as the embedded surfaces bounded by $\gamma_t$ (namely
$K_t(\Sigma_1)$) all have the same area, $K_t(\gamma_0)$ is an exact Lagrangian
isotopy, and so $\gamma_0$ and $\gamma_1$ are Hamiltonian isotopic. \end{proof} 
 
Applying Lemma \ref{curves} to $(\Sigma=\bfc \setminus \text{Critv}(\Pi_n), \sigma)$ where $\sigma$ is as in Lemma \ref{areas}, we get the following:
 
\begin{corollary}\label{isotopy} Suppose embedded circles $\gamma_0$ and
$\gamma_1$ are isotopic in $\bfc \setminus \text{Critv}(\Pi_n)$ and
$\tau_{\gamma_0} = \tau_{\gamma_1}$. Then $\f{T}_{n, \gamma_0}$ and
$\f{T}_{n, \gamma_1}$ are Hamiltonian isotopic in $S_n$.  \end{corollary}
\begin{proof} By Lemma $\ref{curves}$, $\gamma_0$ and $\gamma_1$ are
Hamiltonian isotopic, let $\gamma_t$ be the image of $\gamma_0$ in such an isotopy then by Lemma $\ref{areas}$ the Lagrangian isotopy of the corresponding matching tori $\f{T}_{n,\gamma_t}$ is exact.  \end{proof}

We also have the following obvious observation.

\begin{corollary} Suppose an embedded circle  $\alpha$ is entirely contained inside
an embedded circle $\beta$. Then $\tau_\alpha < \tau_\beta$.  \end{corollary}

\begin{remark}\label{round} A direct computation shows that if we take
$\gamma_r$ to be a circle of radius $r>1$ centred at the origin,
$\tau_{\gamma_r}$ approaches $m=\pi+ n+1$ as $r$ approaches $1$. As $\tau_r$
grows to infinity when $r$ grows,  any $\tau$ above $m$ can be obtained by
taking a circle of some unique radius.
\end{remark}

We next complete the computation of Maslov index on $\pi_2(S_n, \f{T}_n)$.

\begin{lemma}\label{mas} For $\beta \in \pi_2(S_n, \f{T}_n)$, the Maslov index $\mu(\beta)
= 2 ( \beta \cdot \Pi_n^{-1}(0))$.  \end{lemma}

\begin{proof} Since Maslov index is invariant under Lagrangian isotopy, it
suffices to prove this formula for the matching tori above round circles
$\gamma_r$ (of any radius bigger than 1).  To this end, we construct a complex
meromorphic volume form $\Omega$ which is nowhere vanishing and has a pole of
order 1 along the divisor $D=\Pi_n^{-1}(0)$ in $S_n$, and with respect to
which $\f{T}_n$ is a special Lagrangian submanifold, i.e.
$\operatorname{Im}(\Omega)|_{\f{T}_n}=0$. Then \cite[Lemma 3.1]{Aur} states
that the Maslov index $\mu(\beta)$ is twice the algebraic intersection number
of $\beta$ with the divisor of $\Omega^{-1}$, that is with $\Pi_n^{-1}(0)$.

Such an $\Omega$ can be obtained by the restriction to $S_n$ of $\hat{\Omega}=
\frac{dx\wedge dy}{2xy-1}$ on $\bfc^3$.  Note that $S_n$ is cut out by the
equation $2xy-1=z^{n+1}$, hence on $S_n$ we have $2xdy+2ydx = (n+1) z^n dz$, so
that  $\frac{dx\wedge dy}{2xy-1}= \frac{(n+1)}{2yz} dz \wedge dy =
-\frac{(n+1)}{2xz} dz \wedge dx$. We see that on $S_n\setminus D$ the form
$\Omega$ is non-vanishing, and since $D$ is given by $z=0$ (and hence both
$x$ and $y$ are non-zero on $D$), $\Omega$ blows up to order 1 at $D$, as
wanted.

It remains to show that the round $\f{T}_n$ are special Lagrangian for $\Omega$. This is the same as in \cite[Proposition 5.2]{Aur}. Namely, we take the Hamiltonian function on $S_n$ given by $H(x,y,z)= |2xy-1|^2$ and consider its Hamiltonian vector field $X_H$. It is symplectically orthogonal to vertical tangent vectors because $H$  is constant on the fibres of $\Pi_n$ and is tangent to the level sets of $H$, that is to the fibres. So $X_H$ is the horizontal lift of the tangent vector of $\gamma_r$, and so is tangent to  $\f{T}_n$.  The tangent space to $\f{T}_n$ is spanned by $X_H$ and a vector field tangent to the vanishing cycle, say $\xi=(ix, -iy)$. Since $\iota_\xi \Omega = \frac{ixdy+iydx}{2xy-1}= i d \log (2xy-1)$, we get $\operatorname{Im} (\xi, X_H)=d \log |2xy-1|(X_H)$ which is zero because $X_H$ is tangent to the level set of $H$. Hence $\f{T}_n$ is special Lagrangian.  \end{proof} 

Recall that the minimal Maslov number of a Lagrangian $L$ in a symplectic
manifold $M$ is defined to be the integer $N_L := \text{min} \{ \mu(A) > 0 | A
\in \pi_2(M,L) \}$ where $\mu(A)$ is the Maslov index.

We summarize the above discussion as: 

\begin{proposition} $\f{T}_n$ is a monotone Lagrangian torus in $(S_n,
d\theta)$, that is \[ 2 \ \omega(A) = \tau \mu(A) \] for any $A \in \pi_2(S_n,
\f{T}_n)$ where $\tau >0$ is the monotonicity constant, a fixed real number
depending only on the path $\gamma$. The minimal Maslov index $N_{\f{T}_n}=2$.
\end{proposition} 

For $n=0$, we get a monotone Lagrangian torus $\f{T}_0 \subset \f{C}^2$, which
is indeed the Clifford torus (\cite{Aur}) and for $n=1$, we have a monotone
Lagrangian torus $\f{T}_1 \subset T^* S^2$. The following proposition shows
that one of the  $\f{T}_1$ is Hamiltonian isotopic to Polterovich's construction of a
monotone Lagrangian torus in $T^* S^2$ (see \cite{AF}).  

\begin{proposition}  \label{ep} Under a  symplectomorphism identifying  $S_1$
and $T^* S^2$, Polterovich torus becomes one of the $\f{T}_{1}$.
\end{proposition}

\begin{proof} $S_1$ is given by  $\{z^2+ 2xy =1\}$ in $\bfc^3$. Under the exact
symplectomorphism $z=z_0, x=\frac{1}{\sqrt2} (z_1+i z_2), y=
\frac{1}{\sqrt2}(z_1-iz_2)$ it is taken to $C=\{(z_0, z_1, z_2)\in \bfc^3|
z_0^2+z_1^2+z_2^2=1  \}$.  We take $s_j$ and $t_j$ to be real and imaginary
parts of $z_j$ respectively, so that  $z_j=s_j+ i t_j$, and let $s=(s_0,s_1,
s_2)$ and $t=(t_0,t_1, t_2)$.  Note that the equations for $S_1$ are
$|s|^2-|t|^2=1$ and $\langle s,t \rangle=0$.

Further,  we take $T^* S^2= \{v \in \bR^3, u \in \bR^3 | \langle v,u \rangle
=0, |v|=1 \}$. It is exact symplectomorphic to  $C$ via  the map  $(v=s
|s|^{-1},  u= t |s|)$.

For a complex number written in polar form $r e^{i\theta}$ we say that $\theta$
is its phase and we note that a vanishing cycle over the point $z=z_0$ given by
$|x|=|y|$ can be alternatively described by the condition that phase of $z_1$
is equal to phase of $z_2$ modulo $\pi$.

The Polterovich torus $T$ is the geodesic flow of unit covectors over $(1,0,0)$
(\cite{AF}). Namely, let $v=(v_0, v_1, v_2)$ be a point in $S^2$. If $v$ is
neither the north nor the south pole, there are exactly two points in $T$
projecting to $v$. To find them, denote $\vec{r}= (v_1, v_2)$, $r=|\vec{r}|$,
so that $v=(v_0, \vec{r})$. Then the cotangent vectors in the torus $T$
projecting to $v$ are $u=(-r, \frac{v_0}{r} \vec{r})$ and $-u=(r,
-\frac{v_0}{r} \vec{r})$.

Let's find the coordinates  $(s, t)$ of the point in $S_1$ corresponding to
$(v, u)$. Since $|u|=|s||t|=1$, we have $|s|^2-|s|^{-2}=1$, so $|s|=
\sqrt{\frac{\sqrt{5}+1}{2}}$, so that $s = \sqrt{\frac{\sqrt{5}+1}{2}} v$ and
$t=\sqrt{\frac{\sqrt{5}-1}{2}} u$. Note that this means $z_1$ and $z_2$ have
the phases that are either equal (if $v_1$ and $v_2$ have the same sign), or
differing by $\pi$.  The point $(\hat{z}_1, \hat{z}_2)$ corresponding to $(v,
-u)$ has real part $\sqrt{\frac{\sqrt{5}+1}{2}} v$ and imaginary part
$-\sqrt{\frac{\sqrt{5}-1}{2}} u$, and $\hat{z}_1$ and $\hat{z}_2$ also have
phases equal or differing by $\pi$.   As $(v_1, v_2)$ varies over a circle, the
points $(v, u)$ and $(v, -u)$ trace out the vanishing cycles over
$z_0=(\sqrt{\frac{\sqrt{5}+1}{2}} v_0, \sqrt{\frac{\sqrt{5}-1}{2}}  r)$ and
$\hat{z}_0=(\sqrt{\frac{\sqrt{5}+1}{2}} v_0, -\sqrt{\frac{\sqrt{5}-1}{2}}  r)$.
We note that the circles that are intersections of the Polterovich torus with
cotangent fibres over the north and south poles are vanishing cycles over the
points $(\sqrt{\frac{\sqrt{5}+1}{2}}, 0)$ and $(-\sqrt{\frac{\sqrt{5}+1}{2}},
0)$.

Hence the Polterovich torus is in fact the union of vanishing cycles over the ellipse 
$z=(\sqrt{\frac{\sqrt{5}+1}{2}} v_0, \sqrt{\frac{\sqrt{5}-1}{2}}  r)$. Note
that $v_0^2+ r^2=1$ means that the curve over which we have the matching torus
is the ellipse focal at $\pm1$ and with eccentricity
$\sqrt{\frac{\sqrt{5}-1}{2}}$.  This curve can be lifted to the torus $T$ as
before with $x(t)=y(t)=\frac{1}{\sqrt{2}} (1-\gamma(t)^{2})$, which can be
computed to be the same ellipse scaled down by $\sqrt{2}$, hence both $x$ and
$y$ projections enclose area $\frac{\pi}{2}$, giving the monotonicity constant
$\pi+\frac{\pi}{2}+\frac{\pi}{2}=2\pi$, as expected. \end{proof}

\begin{remark} As $2\pi > \pi+2$, Remark~\ref{round} and Lemma~\ref{isotopy}
imply that $T$ is Hamiltonian isotopic to one of the round matching tori.
\end{remark}

\subsection{Floer cohomology of matching tori}

Since $\f{T}_n$ is monotone of minimal Maslov index 2, its self-Floer
cohomology is well-defined and can be computed using the pearl complex. This
complex was first described by Oh in $\cite{ohpearl}$ (see also Fukaya
\cite{Fukaya}) and was studied extensively in the work of Biran and Cornea
(see \cite{BC0}, \cite{BC1} for detailed accounts).  

Before proceeding to the computation proper, we shall give a brief review of
the pearl complex.  We generally follow \cite{BC0} and \cite{BC2} to which the
reader is referred for details, however we will adapt the conventions of Floer \emph{cohomology}, rather than Floer homology (see also
\cite{sheridan}). 

Given a monotone Lagrangian $L$ inside a geometrically bounded symplectic
manifold $M$ (Stein manifolds in particular are geometrically bounded, see
\cite[Section 2]{CGK} for a definition and discussion), the pearl complex of
$L$ is a deformation of its Morse complex by quantum contributions coming from
holomorphic disks with boundary on $L$. In order to define this complex, we
take the coefficient ring to be the Laurent polynomials
$\La=\f{Z}_2[t,t^{-1}]$, fix a Morse function $f$ on $L$ with set of critical
points $\Crit(f)$, a Riemannian metric $\rho$ on $L$, and an almost-complex
structure $J$ on $M$ compatible with our symplectic form $\omega$. The pearl
complex has the underlying vector space $\mathcal{C}^* (L;f,\rho,J)=(\f{Z}_2
\langle \Crit(f) \rangle \otimes \Lambda)$, which inherits a relative
$\f{Z}$-grading coming from the Morse index grading on  $\f{Z}_2 \langle
\Crit(f)\rangle$ and the grading given on $\La$ by $\operatorname{deg} t= N_L$.

We define a differential on $\mathcal{C}^*(L;f,\rho,J)$ by counting pearls -
sequences of gradient flowlines of $f$ interspersed with holomorphic disks.
Namely, denote by  $\Phi_t$, $-\infty \leq t \leq \infty$ the gradient flow of
$(f, \rho)$. Given a pair of points $x, y \in L$ and a class $0 \neq A \in
H_2(M,L)$ consider for all $l \geq 0$ the sequences $(u_1, \ldots, u_l)$ of
non-constant J-holomorphic maps $u_i : (\f{D}, \partial \f{D}) \rightarrow (M,
L)$ with \begin{enumerate}

\item gradient trajectory of possibly infinite length $t'$ from $x$ to $u_1$
i.e.  $\Phi_{t'}(x)=  u_1(-1)$ \item gradient trajectories of length $t_i$
between $u_i$ and $u_{i+1}$ i.e. $\Phi_{t_i}(u_i(1)) = u_{i+1}(-1)$ \item
gradient trajectory of possibly infinite length $t''$ from $u_l$ to $y$  i.e.
$\Phi_{t''}(u_l(1))= y$ \item $[u_1] + \cdots [u_l] = A$ \end{enumerate} Two
such sequences $(u_1, \ldots, u_l)$ and $(u'_1, \ldots, u'_{l'})$ are
equivalent if $l=l'$ and each $u'_i$ is obtained from $u_i$ by precomposing
with holomorphic automorphism of $\f{D}$ that fixes $1$ an $-1$.  We define the
moduli space $\mathcal{P}_{\textnormal{prl}}(x,y;A;f,\rho,J)$ to be the space
of such sequences modulo equivalence.  In addition, for $A=0$  we define
$\mathcal{P}_{\textnormal{prl}}(x,y;A;f,\rho,J)$ to be the space of
unparametrized trajectories of the gradient flow $\Phi_t$ from $x$ to
$y$.  If $x$ and $y$ are critical points of $f$, then the expected dimension
$\delta_{\textnormal{prl}}(x,y;A)$ of
$\mathcal{P}_{\textnormal{prl}}(x,y;A;f,\rho,J)$ is $|y| - |x| + \mu(A)-1$. 

\begin{theorem}[{c.f. \cite[Theorem 2.1.1]{BC0}}] 

For a generic choice of the triple $(f,\rho,J)$ we have:

\begin{itemize} 

\item For all $x,y \in \Crit(f)$ and $A \in  H_2(M,L)$ such that
$\delta_{\textnormal{prl}}(x,y;A)=0$, the moduli space
$\mathcal{P}_{\textnormal{prl}}(x,y;A;f,\rho,J)$ is a finite number of points
and we can define $d(x)=\sum_{y,A} \bigl(  \#_{\mathbb{Z}_2}
\mathcal{P}_{\textnormal{prl}}(x,y;A;f,\rho,J)  \bigr) t^{\frac{\mu(A)}{N_L}}
y$.

\item Extending $d$ to  $\mathcal{C}^*(L;f,\rho,J)$ linearly over $\La$ we get a
chain complex (i.e. $d^2=0$), and the homology of this chain complex is
independent of the choices of $J, f, \rho$.

\item  There is a canonical (graded) isomorphism $H^{\ast}(\mathcal{C}^*(L;f,\rho,J))\to HF^{\ast}(L; \Lambda)$. 
   \end{itemize}

\end{theorem}

\begin{remark} Let us make a remark on the requirement in the above theorem that the triples
$(f,\rho,J)$ be generic. What we require is to have the pair $(f, \rho)$ be
Morse-Smale, making all stable and unstable manifolds of $f$ transverse and
hence making the moduli spaces of gradient trajectories smooth, and to have $J$
that makes moduli of holomorphic disks with boundary on $L$ and two boundary
marked points regular; in addition we require all the evaluation maps from the moduli
spaces above into our symplectic manifold $M$ to be transverse in tuples, so
that the corresponding moduli spaces
$\mathcal{P}_{\textnormal{prl}}(x,y;A;f,\rho,J)$  are transversally cut out.
Note that only the moduli spaces that appear in building
$\mathcal{P}_{\textnormal{prl}}(x,y;A;f,\rho,J)$ with expected dimension
$\delta_{\textnormal{prl}}(x,y;A) \leq 1$ need to be regular. What we will use
in our computation is a complex structure that is regular for disks of Maslov
index 2 and a generic Morse-Smale function; this is sufficient for a monotone
two-dimensional Lagrangian torus.
\end{remark}
\begin{remark} $HF^*(L; \Lambda)$ is a unital (associative) ring with a
relative $\f{Z}$-grading, where the ring structure is given by counting
pseudoholomorphic triangles of Maslov index zero. We can indeed fix an absolute
$\f{Z}$-grading by requiring that the unit lies in $HF^0(L; \Lambda)$.
Similarly, the relative $\f{Z}$-grading at the chain level
$\mathcal{C}^*(L;f,\rho,J)$ can be upgraded to an absolute $\f{Z}$-grading by
requiring that the generators of Morse index $0$ lie in degree 0. Finally,
note that equivalently we could have worked with $HF^*(L; \f{Z}_2)$ by setting
$t=1$ in the definition of the chain complex. We then only get $\f{Z}/
N_L$-grading. On the other hand, $HF^*(L;\f{Z}_2)$ and $HF^*(L; \Lambda)$ carry
the same information since $HF^*(L; \Lambda)$ is $N_L$ periodic in the sense
that $HF^{*+N_L}(L;\Lambda) = t \cdot HF^*(L;\Lambda)$.  \end{remark} 
\begin{remark} The Lagrangians $\f{T}_n$ are tori, hence they are orientable and can be equipped
	with spin structures. This would allow us to take $\Lambda= \f{Z}[t,t^{-1}]$ as
	our coefficient ring. Doing so would require picking orientations and
	spin structures on $\f{T}_n$, and paying attention to the induced
	orientations of moduli spaces of discs in Floer cohomology
	computations. We avoid this refinement as it is not needed for our application.
\end{remark}
Let $L$ be a monotone Lagrangian in $(M,\omega)$ with minimal Maslov number
$N_L \geq 2$. In this case, following Biran and Cornea (Section 6.1.1 \cite{BC0}), we define a homology class $c(L) \in
H_1(L;\f{Z}_2)$ as follows: Let $J$ be an $\omega$-compatible almost complex
structure such that all the holomorphic disks of Maslov index $2$ are
regular (call such $J$ regular). Monotonicity ensures that there are only
finitely many homology classes in $H_2(M,L)$ represented by a holomorphic disk
and an application of a lemma of Lazzarini (\cite{lazzarini}) shows that all such disks are
simple. Thus, the set of regular $J$ is of second category in the space of
compatible almost complex structures. Pick a (generic) point $p \in L$ such
that the number of Maslov index $2$ holomorphic disks $u: (\f{D}, \partial
\f{D}) \to (M,L)$ with $p \in u(\partial \f{D})$ is finite, call this number
$l$. Then the boundaries of these holomorphic disks represent homology classes (counted with multiplicity) 
$c_1, \ldots, c_l \in H_1( L; \bfz_2)$ and the homology class $c(L)$ is simply the
sum $c(L) = \sum_{i=1}^l c_l$. Standard cobordism arguments show that $c(L)$ is independent of $J$ and $p$. 

The pearl complex model for self-Lagrangian Floer cohomology admits a degree
filtration as follows: $\mathcal{F}^k (\mathcal{C}^* (L;f,\rho,J)) =(\f{Z}_2
\langle \Crit(f) \rangle \otimes \mathcal{F}^k (\f{Z}_2[t,t^{-1}]))$, where
$\mathcal{F}^k \f{Z}_2[t,t^{-1}] = \{ P \in \f{Z}_2[t,t^{-1}] | P(t) = a_k t^k
+ a_{k+1} t^{k+1} + \ldots \}$. The differential clearly respects this
filtration and the degree preserving part corresponds to pearly trajectories
with $A=0$, which are indeed Morse trajectories. Therefore, one obtains a
spectral sequence from $H^{\ast}(L;\f{Z}_2)$ to $HF^*(L;\f{Z}_2)$ (this is
known as Oh's spectral sequence \cite{oh}). Biran and Cornea's careful analysis
of the algebraic structure of this spectral sequence shows that in our situation the class $c(\f{T})$ completely determines the Floer cohomology $HF^*(\f{T}; \bfz_2)$ additively, which we record as follows:
 \begin{proposition}[{\cite[Proposition 6.1.4]{BC0}}] \label{toriQH}
Let $\f{T}$ be a monotone Lagrangian 2-torus in a symplectic 4-manifold
$(S, \omega)$ with minimal Maslov number $N_{\f{T}} \geq 2$ .\\ If $c(\f{T}) = 0$, then $HF^* (\f{T}; \bfz_2) \simeq H^* ( \f{T}; \bfz_2)$ (as $\bfz_2$-graded vector spaces).\\ Conversely, if $c(\f{T}) \neq 0$, then $HF^* (\f{T}; \bfz_2)=0$.  
\end{proposition}
\begin{remark} If we only wanted to show $HF^*(L;\f{Z}_2) \neq 0$ when
$c(L)=0$, we could argue as follows: We pick a Morse function $f$ on $L$
with a unique maximum, call it $m$. Since it represents a generator for
$H^2(L)$, for degree reasons, it will survive in $HF^*(L; \f{Z}_2)$ if
$\partial (m)=0$. On the other hand, the knowledge of Maslov index $2$
disks through $m$ allows us to compute $\partial(m) = PD(c(L) ) \cdot t $
where $PD(c(L) )$ is a chain consisting of linear combinations of index $1$
critical points of $f$ representing the Poincar\'e dual of $c(L) \in
H_1(L)$. Hence, if $c(L) =0$, then $m$ represents a non-trivial class in $HF^*(L;
\f{Z}_2)$. Proposition \ref{toriQH} shows that when $L$ is a torus, this is actually equivalent to $HF^*(L;\f{Z}_2) \simeq H^*(L; \f{Z}_2)$ \end{remark}

In view of Proposition \ref{toriQH}, we determine the Floer cohomology of the tori $\f{T}_n$ via a calculation of $c(\f{T}_n)$.
\begin{lemma} \label{sections} $c(\f{T}_0) \neq 0$ and $c(\f{T}_n)=0$ for $n>0$.
\end{lemma}
\begin{proof}
Recall that if $\f{T}_n = \f{T}_\gamma$ is a
matching torus over a curve $\gamma$ and $\gamma$ bounds the disc $D_\gamma
\subset \bfc$, then for a $J$ making $\pi_n$ holomorphic, by maximum principle,
the sections have to project to $D_\gamma$. So, we will not distinguish $\Pi_n
: S_n \mapsto \f{C}$ and its restriction $\Pi_n| {\Pi_n^{-1}(D_\gamma)} \mapsto
D_\gamma$ when counting holomorphic sections of $\Pi_n$ with boundary on
$\f{T}_\gamma$. 

First we note that, if we have an isotopy of $\gamma_t$ in $\bfc\setminus
\text{Critv}(\Pi_n)$ such that $\tau_{\gamma_t} = \text{const}$, so that the
isotopy lifts to a Hamiltonian isotopy $F_t$ of $\f{T}_{\gamma_t}$ bounding
discs $D_t$, as in Corollary \ref{isotopy}, then $F_t$ gives identifications of
all $H_1(\f{T}_{\gamma_t}, \bfz_2)$ and the moduli spaces of sections $\Pi_n|
\Pi_n^{-1}(D_t) \mapsto D_t$ with boundary on $\f{T}_{\gamma_t}$ representing a
given class $a \in H_1(\f{T}_{\gamma_t}, \bfz_2)$ of (minimal) Maslov index 2
are cobordant. This can be seen as follows: For fixed $t_0 < t_1$ let $J_{t_0}$ and $J_{t_1}$ be regular almost complex structures making the moduli spaces of Maslov index $2$ sections $\mathcal{M}_{t_i} = \mathcal{M}(\Pi_n, \f{T}_{\gamma_{t_i}}, J_{t_i})$ regular one-dimensional manifolds. Consider the space $\mathcal{J}$ of almost complex structures in the total space which are simultaneously regular for counting sections in
$\mathcal{M}(\pi_r,\f{T}_{\gamma_t})$ for all $t\in [t_0,t_1]$ and makes $\Pi_n$
holomorphic, this is a subset of second category in the space of almost complex
structures on $S_n$ (since $[t_0,t_1]$ is compact and it is of second category for a fixed $t$). A generic path $J_t$ of almost complex structures in this space connecting $J_{t_0}$ and $J_{t_1}$ gives a smooth cobordism $\mathcal{M}_t = \mathcal{M}(\Pi_n, \f{T}_{\gamma_t}, J_t)$ of moduli spaces $\mathcal{M}_{t_0}$ and $\mathcal{M}_{t_1}$ since at no point during the isotopy $\f{T}_{\gamma_{t}}$ bounds Maslov index $\leq 0$ disks. Furthermore, since the matching tori $\f{T}_{\gamma_t}$ are parallel transported to each other, we get a bordism of the images of the  evaluation maps $ev_t : \mathcal{M}_t \to \f{T}_{\gamma_t}$ by considering the parametrized evaluation map $ev : [t_0, t_1] \times \mathcal{M}_t \to
\f{T}_{\gamma_{t_0}}$ where we use the parallel transport to identify $\f{T}_{\gamma_t}$ with $\f{T}_{\gamma_{t_0}}$.  Therefore, for the purpose of algebraically counting of pseudoholomorphic sections of $\Pi_n$ with boundary on $\f{T}_\gamma$ we are free to move $\gamma$ with such an isotopy.  

Now, consider the deformation $\gamma_t$ as in Figure \ref{deform} where
$\gamma_0=\gamma$ and $\gamma_1 = \alpha \# \beta$  such that $\alpha$ is an
embedded circle that encloses only one critical value and $\beta$ encloses the
remaining $n$ critical values.  

\begin{figure}[!h] \centering
\includegraphics[scale=0.6]{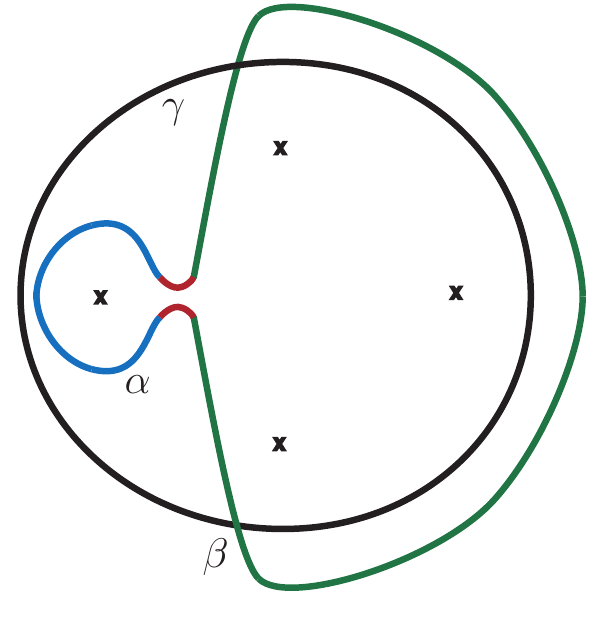}  \caption{Deformation of $\gamma$}
\label{deform} \end{figure}

To be precise, $\alpha$ and $\beta$ are closed embedded circles in the base
$\f{C}$ that intersect at a unique point $p$ and $\gamma_t$ is an isotopy
through embedded curves with $\tau_{\gamma_t} =\text{const.}$ and $\gamma_1$ is
very close to $\alpha \vee \beta$ (we can find such an isotopy by Lemma
\ref{areas}).  Let $D_\alpha$ , $D_\beta$ be the disks that $\alpha$ and
$\beta$ bound.  Then we can consider the Lefschetz fibrations $\pi_{\alpha}=
\Pi_n| {\Pi_n^{-1}(D_\alpha)} \mapsto D_\alpha$ and $\pi_{\beta}= \Pi_n|
{\Pi_n^{-1}(D_\beta)} \mapsto D_\beta$. We also set $\pi_{\gamma_1} = \Pi_n|
{\Pi_n^{-1}(D_{\gamma_1})} \mapsto D_{\gamma_1}$. Now, we can deform $\gamma_1$
to a sufficiently close neighborhood of $\alpha \vee \beta$ so that the
Lefschetz fibration $\pi_{\gamma_1}$ is a boundary sum of the Lefschetz
fibrations $\pi_{\alpha}$ and $\pi_{\beta}$ and the matching torus
$\f{T}_{\gamma_1}$ is obtained as a connect sum $\f{T}_{\alpha} \#
\f{T}_{\beta}$. To be careful, one first deforms the symplectic structure in a
neighborhood of the fibre above $p$, so that it is a trivial symplectic bundle
$F \times [-1,1]^2$ where $F$ is the fibre, and the piece of the Lagrangian
torus over $\gamma_1$ becomes two trivial circle bundles over the intervals $\{
	\pm \epsilon \} \times [-1,1]$. One then surgers the Lagrangian boundary condition within this trivialization so that the outcome
	is $\f{T}_\alpha$ and $\f{T}_\beta$.  (For more details, we refer to
	Proposition 2.7 and the preceding discussion in \cite{LES} for the
	details of boundary sum of Lefschetz fibrations which carry Lagrangian
	boundary conditions).

Now, let $\mathcal{M}(\pi_\alpha, \f{T}_\alpha, J_\alpha)$ be the moduli space
of $J_\alpha$ holomorphic sections of $\pi_\alpha$ with boundary condition
$\f{T}_\alpha$ and similarly let $\mathcal{M}(\pi_\beta, \f{T}_\beta, J_\beta)$
be the corresponding moduli space for $\beta$. Let $V = \f{T}_\alpha \cap
\f{T}_\beta$ is the vanishing cycle on the fibre $\Pi_n^{-1}(p)$ and,
$ev_\alpha : \mathcal{M}(\pi_\alpha, \f{T}_\alpha, J_\alpha) \to V$ and
$ev_\beta : \mathcal{M}(\pi_\alpha, \f{T}_\beta, J_\beta) \to V$ are evaluation
maps.  The basic gluing theorem [\cite{LES} , Proposition 2.7] proves that if
$J_\alpha$ and $J_\beta$ are regular and $ev_\alpha$ and $ev_\beta$ are
mutually transverse, then there exists a complex structure $J$ so that
$\mathcal{M}(\pi_{\gamma_1}, \f{T}_{\gamma_1} , J)$ is regular and is given as
a fibre product of \[ \mathcal{M}(\pi_{\gamma_1}, \f{T}_{\gamma_1} , J)^k=
\bigsqcup_{p+q-1  =k} \mathcal{M}(\pi_\alpha, \f{T}_\alpha, J_\alpha)^p
\times_{V} \mathcal{M}(\pi_\beta, \f{T}_\beta, J_\beta)^q \] where the fibre
product is taken with respect to the evaluation maps $ev_\alpha$ and $ev_\beta$
and the superscripts are dimensions. Recall that we are interested in counting
Maslov index 2 disks with boundary on $\f{T}_{\gamma}$, which live in the moduli space of index $\mu +
\text{dim}(\f{T}_\gamma) -3 =1$. Hence, according the gluing result above, it
suffices to understand the Maslov index 2 disks for $\f{T}_\alpha$ and
$\f{T}_\beta$.  

In fact, by induction it suffices to understand only the base case
$\mathcal{M}(\pi_\alpha, \f{T}_\alpha, J_\alpha)$, i.e. when only one critical
point is enclosed, since if $\beta$ encloses more than one critical point, we
can apply the above deformation to $\beta$ separately to break it up into
smaller pieces until each piece encloses only one critical point, see Figure
\ref{induction} for an illustration.

\begin{figure}[!h] \centering \includegraphics[scale=1]{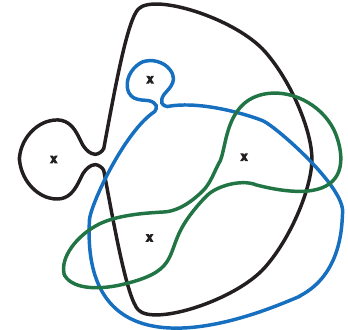}
\caption{Inductive deformations} \label{induction} \end{figure}

To tackle the base case, one applies a degeneration argument due to Seidel
\cite[Section 2.3]{LES}. Namely, for convenience, we can assume by an isotopy
through curves with fixed $\tau = \tau_\alpha$, $\alpha$ is
a round circle by the argument in the beginning of the proof. Let $D_r$ be disks of radius $r \in (0, a]$ where $D_a =
D_\alpha$. One considers the restrictions $\pi_r : \pi_\alpha|
\pi_\alpha^{-1}(D_r) \to D_r$ with Lagrangian boundary conditions given by
matching tori $\f{T}_r$ above $\partial{D_r}$. Exactly the same cobordism argument that we have given in the beginning of the proof shows that the algebraic count of pseudoholomorphic sections of $\pi_\alpha$ with boundary $\f{T}_r$ does not depend on $r$ when we vary $r$ in a compact interval (Note that again, none of the tori $\f{T}_r$ bound a Maslov index $\leq 0$ holomorphic disk). Now, Seidel proves a compactness result when as one lets $r \to 0$ (\cite[Lemma 2.15]{LES}) to conclude that when $r$ is sufficiently small, the
moduli space of sections can be computed using a model Lefschetz fibration.
$\pi \colon \f{C}^2 \to \f{C}$ given by $\pi:(x_1, x_2) \in \bfc^2\mapsto
(x_1^2+x_2^2) \in \bfc$ as in [\cite{LES}, eq.  (2.18)]. Lemma 2.16 of \cite{LES} explicitly computes all
sections of $\pi$ with boundary on $\f{T}_{\gamma_r} \subset \bfc^2$, where
$\gamma_r$ is the round circle of radius $r$ centered at $0$.  These are the
maps from closed disk of radius $s$ to $\mathbb{C}^2$ given by  \[ u_{a, \pm}
(w) = (r^{-1/2} aw + r^{1/2}\bar{a} , \pm i(r^{-1/2} aw-r^{1/2} \bar{a}) ) \]
for $a\in \f{C}$ with $|a| = \frac{1}{2}$ so this space is diffeomorphic to
	$S^1 \sqcup S^1$. Moreover by the same lemma this moduli space of
	sections is regular (for the standard complex structure on $\f{C}^2$).

 Following \cite{Aur}, we note that $\f{T}_{0, \gamma_r} \in \bfc^2$ is in fact
 the Clifford torus $|x|=|y|=r^{\frac{1}{2}}$. From the above explicit
 description, the images of the boundaries of the two families of holomorphic
 disks on $\f{T}_0$ are given by $x=const.$ and $y=const.$ (The same two
 families were obtained as the outcome of the computation in \cite{Cho},
 Theorem 10.1. This again shows that these disks are regular by \cite{Cho},
 Theorem 10.2.) Note that there are exactly two holomorphic sections with
 boundary through any given point $p \in \f{T}_0$ and their boundaries
 intersect transversely at a single point.  Therefore, the homology classes in
 $H_1(\f{T}_0)$ represented by the boundaries of these two families are of the
 form $L$ and $L+V$, where both $L$ and $L+V$ project to the generator of
 $H_1(S^1)$ under $\Pi$ and $V$ is the class of the vanishing cycle. 

Thus, we have determined $\mathcal{M}(\pi_\alpha, \f{T}_\alpha, J_\alpha)$
where $\alpha$ encloses only one critical point and $J_\alpha$ is the standard
complex structure (which is regular) and the computation also gives the evaluation map $ev_\alpha$. It remains to perform the inductive step of the computation to compute the Maslov index $2$ sections of $\f{T}_\gamma$

As discussed above the gluing theory shows that the count of sections for
$(S_n, \f{T}_\gamma)$ can be understood as the count of sections for the
$n$-fold boundary connect sum of $(S_0, \f{T}_0)$, which will be denoted by $(\Sigma_n, \tau_n)$. To describe the holomorphic disks in it we need a basis for
 $H_1(\tau_n)$. One element of the basis can be taken to be the vanishing cycle
 $V_n$. The choice of a second basis element is obtained a posteriori by the following lemma.

 \begin{lemma} Through any point on $\tau_n$ there are $2^{n+1}$ disks of
 Maslov index 2 in $(\Sigma_n, \tau_n)$, and there exist elements $L_n \in
 H_1(\tau_n)$ which together with $V_n$ form a basis of $H_1(\tau_n)$ and such
 that there are ${{n+1}\choose{k}}$ disks with boundary class $kV_n+ L_n$.
 \end{lemma} \begin{proof} We prove this by induction. As discussed
 above, the base case is the Clifford torus, where there are indeed 2
 holomorphic disks through every point, in classes whose difference is $V_1$.
 These moduli spaces are regular for standard complex structure on $\bfc^n$.
 Call one of them $L_1$ and another $L_1+V_1$. The inductive step is given by
 using the Seidel's gluing formula  \cite{LES}, Proposition 2.7 together with
 the base case. Each of the $2^n$  disks given by induction hypothesis in
 $(\Sigma_{n-1}, \tau_{n-1})$ glues to either of the 2 disks in $(S_0,
 \f{T}_0)$.  By the same proposition, the glued up moduli spaces are regular.
 Denoting by $L_{n}$ the class in  $H_1(\tau_n)$ obtained by gluing $L_{n-1}$
 and $L_1$, we have that the number of disks with boundary in class $k V_n+
 L_n$ is ${n\choose k} + {n\choose {k-1}} = {{n+1}\choose{k}}$, as claimed.
 \end{proof}

\emph{Completion of the proof of Lemma \ref{sections} }: We conclude that the total boundary class of Maslov index $2$ disks is given by
$\sum\limits_{k} {{n+1}\choose k} (k V_n+L_n)= 2^{n+1}L_n+ (n+1) 2^{n} V_n$,
which is $0 \in H_1(\tau_n; \bfz_2)$ for all $n>0$ and $[V_0] \neq 0 \in H_1(\tau_0; \bfz_2 )$ for $n=0$. Since the moduli spaces of
discs in $(\Sigma_n, \tau_n)$ and $(S_n, \f{T}_n)$ are identified, it follows that $c(\f{T}_0) \neq 0$ and $c(\f{T}_n)=0$ for $n>0$. \end{proof}

\begin{proposition} \label{QHup} $HF^{\ast}(\f{T}_n; \bfz_2) \simeq H^{\ast}(
\f{T}_n; \bfz_2)$ for $n>0$.  \end{proposition} 

\begin{proof}
	The proof follows immediately from Proposition \ref{toriQH} together with Lemma \ref{sections}.
\end{proof}

\begin{rem} In view of Proposition~\ref{ep},
Proposition~\ref{QHup} generalizes a theorem of Albers and Frauenfelder from \cite{AF} where the authors computed $HF^*(\f{T}_1; \f{Z}_2)$.
\end{rem}

\section{The rational homology balls $B_{p,q}$}

\subsection{A finite group action on the $A_n$ Milnor fibre}

As before, let $p>q >0$ be two relatively prime integers. Let $\f{Z}_p = \{ \xi
\in \f{C} : \xi^p =1 \} $ be the cyclic group. Let us consider a
one-parameter smoothing of the isolated surface singularity of
type $A_{p-1}$, i.e. we consider the hypersurface singularity given by $ z^p+2xy
= 0 \subset \f{C}^3$ and the smoothing of this singularity given by $F :
\f{C}^3 \to \f{C}$ , where $F(x,y,z)= z^p+2xy$. We let $\Gamma_{p,q}$ to denote
the following action of $\f{Z}_p$ on $\f{C}^3$ given by \[\xi \colon (x,y,z)
\to (\xi  x, \xi^{-1} y, \xi^q z ) \]

Clearly, the action is free outside of the origin and the function $F$ is
invariant under the action. Indeed, we get a $\f{Q}$HD-smoothing of the
singularity $F^{-1}(0) /  \Gamma_{p,q}$. The latter is known to be the cyclic
quotient singularity of type $(p^2, pq-1)$ (\cite{wahl} Example 5.9.1). We
denote the Milnor fibre $F^{-1}(1) / \Gamma_{p,q}$ by $S_{p-1}/
\Gamma_{p,q} = B_{p,q} $. 

The action $\Gamma_{p,q}$ can be visualized easily in terms of the Lefschetz
fibration $\Pi : S_{p-1} \to \f{C}$. Namely, $\f{Z}_p$ acts freely by lifting
the rotation of the base of the Lefschetz fibration around the origin by an
angle of $\frac{2\pi q}{p}$, as well as rotating the fibres by an angle of $
\frac{2\pi }{p}$.

Note that this makes it clear that $B_{p,q}$ is a rational homology ball. On
the other hand, since $S_{p-1}$ is simply-connected, we have $\pi_1(B_{p,q}) =
\f{Z}_p$. Note also that the Stein structure on $S_{p-1}$ induces a Stein
structure on $B_{p,q}$. Recall that $B_{p,q}$ is a smoothing of the cyclic
quotient singularity of type $(p^2, pq-1)$, that is $\f{C}^2/ \f{Z}_{p^2}$
where $\f{Z}_{p^2} = \{ \xi \in \f{C} : \xi^{p^2} =1 \} $ acts by $\xi \colon
(w_1,w_2) \to (\xi w_1, \xi^{pq-1} w_2)$. Therefore the boundary of $B_{p,q}$
is the lens space $L(p^2, pq-1)$. The Stein structure on $B_{p,q}$ induces a
contact structure $\xi_{p,q}$ on $L(p^2, pq-1)$, which is also filled by the
singular fibre of the deformation. This can in turn be resolved to obtain a
Milnor filling by the resolution of the cyclic quotient singularity, which we
denote by $C_{p,q}$. $C_{p,q}$ is given by the linear plumbing graph below: 

\begin{figure}[!h]
\centering
\includegraphics[scale=0.8]{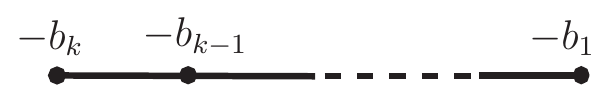}  
\caption{$C_{p,q}$}
\label{Figure1}
\end{figure}

Thus, $C_{p,q}$ is the linear plumbing of disk bundles of over the 2-sphere
with Euler number $-b_i$.  Here $b_i$ are obtained by the unique continued
fraction expansion $\frac{p^2}{pq-1} = [b_k,b_{k-1},\ldots,b_1]$ with all $b_i \geq 2$.

In fact, we claim that the Stein surfaces $B_{p,q}$ are exactly those that are
used by Fintushel-Stern (and J. Park) in rational blow-down operation. Namely,
let $K(m,n)$ denote the 2-bridge knot, whose double branched cover is the lens
space $L(m,n)$. It is known that $K(p^2, pq-1)$ is slice (in fact ribbon) for
$p>q>0$ relatively prime (see for ex. \cite{lisca}). Fintushel-Stern's rational
homology balls (\cite{FS}) are given by the double branched cover of the
four-ball branched over the slice disk for $K(p^2,pq-1)$.

\begin{proposition} $B_{p,q}$ is diffeomorphic to double branched cover of
$D^4$ branched along the slice disk bounding $K(p^2,pq-1)$.  \end{proposition} 

\begin{proof} As we observed above, $C_{p,q}$ yields a Milnor filling of the
contact structure $(L(p^2,pq-1), \xi_{p,q})$. Therefore, by \cite{LO}, the
contact structure $\xi_{p,q}$ must be universally tight. (This also follows
from the fact that $\xi_{p,q}$ is the induced contact structure on the boundary of
the cyclic quotient singularity of type $(p^2, pq-1)$).  Up to contact
isomorphism, it is known that there is a unique universally tight contact
structure on $L(p^2,pq-1)$. Furthermore, Lisca has given a classification
result for the diffeomorphism types of the fillings of the tight contact
structures on lens spaces (\cite{lisca}). It follows from this classification
that in the case of $(L(p^2,pq-1),\xi_{p,q})$, there are two possibilities for
the diffeomorphism types of symplectic fillings, and these classes are realized
by the manifolds $C_{p,q}$ and the double branched cover of $D^4$ branched
along the slice disk bounding $K(p^2,pq-1)$. The latter must then be
diffeomorphic to $B_{p,q}$ since $B_{p,q}$ is a Stein filling which is not
diffeomorphic to $C_{p,q}$. \end{proof}

We have equipped the manifold $B_{p,q}$ with the Stein structure induced from
$S_{p-1}$ given as the finite free quotient of the Stein structure on
$S_{p-1}$. This is the same as the Stein structure on $B_{p,q}$ thinking of it
as an affine algebraic variety because $B_{p,q}$ is an algebraic quotient of
$S_{p-1}$. Note that there exists a unique Stein structure up to deformation on
$S_{p-1}$. This follows, for example, from \cite{wendl}. Therefore, it seems
likely that $B_{p,q}$ in fact has a unique Stein structure, however we do not
know how to prove or disprove this. On the other hand, any putative exotic Stein
structure on $B_{p,q}$ would lift to the standard Stein structure on $S_{p-1}$.
Therefore, for our arguments, we do not need to make precise which Stein
structure is being considered on $B_{p,q}$. Note also that the same reasoning
shows that any Stein structure on $B_{p,q}$ would have to fill the unique (up
to contact isomorphism) universally tight contact structure.

\subsection{Legendrian surgery diagram of $B_{p,q}$}
\label{Leg}

In this section, we construct a Stein structure on $B_{p,q}$ via Legendrian surgery on a Legendrian knot on $S^1 \times S^2$. We see from our description that the $p$-fold cover of the surgery diagram that we depict gives a surgery diagram of the Stein structure on $S_{p-1}$. 

Recall that the Stein structure on $S_{p-1}$ can be drawn as in top figure of
Figure \ref{Leg1} starting from the Lefschetz fibration $\Pi$. It is understood
that all the framings are given by $tb-1$ framing, where $tb$ denotes the
Thurston-Bennequin framing. From the Lefschetz fibration view, the 1-handle can
be understood as the thickening of the fibre over the origin and the
$2$-handles correspond to thimbles over the linear paths connecting the origin to
the critical values ($p^{th}$ roots of unity).

Now, we can apply $q$ full negative twists around the 1-handle, which would
change the smooth framing of individual handles from $-1$ to $-1-q$, and this
can be drawn as in the middle figure of Figure \ref{Leg1}, where there are $p$
twisted handles which have $tb-1=-1-q$, as wanted. In other words, the middle
figure also gives a Stein structure on $S_{p-1}$ and since there is a unique
Stein structure on $S_{p-1}$ up to Stein deformation, we can in fact work with
the middle figure. The advantage of doing this is that it allows us to see the $\Gamma_{p,q}$ action on the diagram. Namely, it sends the 1-handle to the
quotient 1-handle and translates the attaching circles of the 2-handles (in the
horizontal direction as drawn). The bottom figure in Figure $\ref{Leg1}$
depicts the quotient diagram for the action $\Gamma_{p,q}$ on $S_{p-1}$
yielding $B_{p,q}$. (See Section 6.3 \cite{GS} for a discussion of finite
covers of handlebody diagrams). Here, there is a unique 2-handle that passes through the 1-handle $p$ times and it has framing $tb-1 = -pq-1$.

\begin{figure}[!h] \centering
	\includegraphics[scale=1]{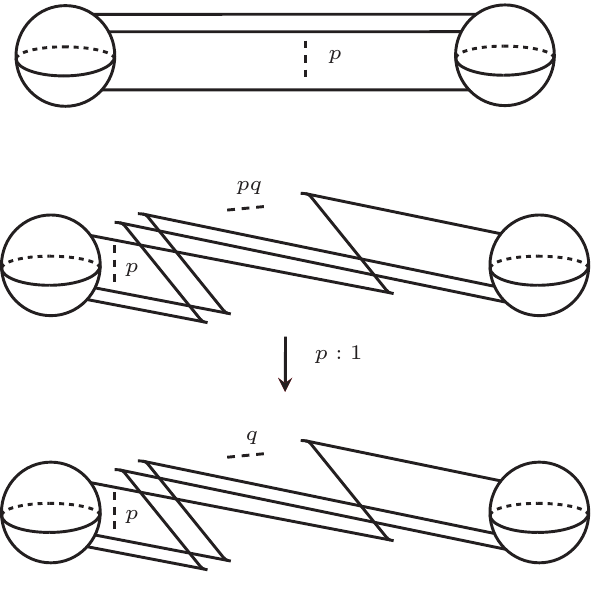}  \caption{Legendrian surgery diagrams: $S_{p-1}$ (top and middle), $B_{p,q}$ (bottom)}
\label{Leg1} \end{figure} 

\begin{remark} The smooth handlebody description of $B_{p,q}$ consisting of just one 1-handle and one 2-handle seems to be not widely known 
	for $q>1$ (see Figure 8.41 in \cite{GS} for $q=1$ which matches with
	the above picture) . Here, we provide not only a smooth handlebody
	description but also a Legendrian realization of the attaching circle
	of the 2-handle as a Legendrian knot in $S^1 \times S^2$ such that the
	smooth framing is given by $tb-1$, hence this description equips
	$B_{p,q}$ with a Stein structure (see \cite{GS} Chapter 11).  From our
	description, it also follows that the Stein structure that we obtain
	this way is the same as the Stein structure induced from $S_{p-1}$ via
	the action $\Gamma_{p,q}$.  \end{remark}

\subsection{Lagrangian submanifolds of $B_{p,q}$}

The exact Lagrangian submanifolds of $S_{p-1}$ has been studied extensively. We will use the understanding provided by Ritter (\cite{ritter}) and Ishii, Ueda and Uehara (\cite{iu}, \cite{iuu}) to prove the following theorem:

\begin{theorem} \label{none} For $p\neq2$, there does not exist any closed exact Lagrangian submanifold in $B_{p,q}$.
\end{theorem}

Before we give a proof of this theorem, we will make it clear what knowledge of
exact Lagrangian submanifolds in $S_{p-1}$ will be needed. In \cite{ritter}, it
is proven using symplectic cohomology with twisted coefficients that every
closed exact Lagrangian submanifolds in $S_{p-1}$ is diffeomorphic to $S^2$. Since the intersection form on $H_2(S_{p-1})$ is negative definite, and the homology class of an orientable closed Lagrangian submanifold $L$ in a Stein surface satisfies $[L] \cdot [L] = - \chi(L)$, it follows easily that orientable closed exact Lagrangians can only be sphere or tori. Ritter's result tells us that in fact any closed exact Lagrangian in $S_{p-1}$ has to be orientable, moreover it cannot be a torus.

Note that there is an abundance of inequivalent exact Lagrangian spheres in
$S_{p-1}$ provided by the matching sphere construction. Ishii, Ueda and
Uehara's results from \cite[Lemma 38]{iuu} (which in turn depends on \cite{iu})  imply that in the exact Fukaya
category of $S_{p-1}$ any spherical object is isomorphic to a matching sphere,
$S_c$ where $c:[0,1] \to \f{C}$ denotes the corresponding embedded path
connecting critical values of the Lefschetz fibration $\Pi : S_{p-1} \to
\f{C}$. More precisely, if $L, L' \subset S_{p-1}$ is an exact Lagrangian
submanifold, we know from \cite{ritter} that they are spheres (which are indeed
spherical objects), hence by \cite{iu}, \cite{iuu} they are isomorphic to
matching spheres $V_c$, $V_{c'}$ where $c, c':[0,1] \to \f{C}$ denote the
corresponding paths. This is useful as it implies that $HF^*(L,L') \simeq
HF(V_c, V_{c'})$. Informally, for the purpose of Floer theory, one can pretend
that every exact Lagrangian submanifold of $S_{p-1}$ is a matching sphere.
Ishii, Ueda and Uehara's result uses homological mirror symmetry to get a
quasi-isomorphic model for the exact Fukaya category of $S_{p-1}$ (this makes use of a formality result proved in \cite{ST}) and uses sheaf theoretical arguments on the mirror category to characterize spherical objects (see also the discussion in \cite[Section 3b]{SAn}). 

\begin{proof}[Proof of Theorem \ref{none}]  Let $L$ be a closed exact
Lagrangian submanifold in $B_{p,q}$.  Then the preimage $L'$ of $L$ in
$S_{p-1}$ by the quotient map is a closed exact Lagrangian submanifold of
$S_{p-1}$ (possibly disconnected). By \cite[Theorem 52]{ritter}, $L' $ is a
union of spheres, and since $L$ is covered by each connected component of $L'$,
$L$ is either a sphere or an $\f{R}P^2$.  A Lagrangian sphere has
self-intersection  $-2$, and hence represents a non-torsion class in
$H_2(B_{p,q})$. This is impossible as $H_2(B_{p,q})=0$.
On the other hand, a Lagrangian $\f{R}P^2$ would have to be double covered by
some number of Lagrangian spheres in $S_{p-1}$. This is an immediate
contradiction if $p$ is odd. 

Suppose $p=2r$ is even. Let $R$ be a generator of the cyclic group $\f{Z}_p$
acting on $S_{p-1}$. Then $L'$ is a disjoint union of $r$ Lagrangian spheres $V,
R(V), \ldots R^{r-1}(V)$ and $R^r$ maps each of these spheres onto themselves
so that the quotient $L$ is an $\f{R}P^2$. 

We now use Ishii, Ueda and Uehara's results from \cite{iu}, \cite{iuu}
discussed above to replace $V$ with an isomorphic object $V_c$ in the exact
Fukaya category of $S_{p-1}$ where $V_c$ is a matching sphere for a possibly
quite complicated path $c$. Now, $R^r$ is the antipodal map, $R^r(x,y,z) = (-x,-y,-z)$. Hence, $R^r V$ is represented by the matching
sphere over the path $-c$.  Since  $R^r V = -V$ this means $V_{-c}$ and $-V$
are isomorphic in the exact Fukaya category, which by  \cite{SK} implies that
$c$ and $-c$ are isotopic (as unoriented paths) by a compactly supported
isotopy in $\f{C}$ that fixes $D=\{ e^{2\pi i k/p} , k=0,1,\ldots p-1 \}$
pointwise. In particular, this implies that if  $c(0) = e^{2\pi i \kappa/p}$
then  $c(1)= - e^{2\pi i \kappa/p}$. 

 Since we assumed $p>2$, $V$ and $R(V)$ are disjoint exact Lagrangian spheres,
 and we have $0 = HF^*(V,R(V)) = HF^*(V_{c}, R(V_{c}))$. Note that $R(V_{c})$
 is simply $V_{c'}$ where $c'(t) = e^{2\pi i q /p} c(t)$ and by \cite[Lemma
 6.14]{SK} , we have that rank of $HF^*(V_{c}, V_{c'})$ is $2 \iota(c,c')$ where
 $\iota(c,c')$ is the geometric intersection number, i.e.  minimal possible
 number of intersections among representatives of the isotopy class of $c$ and
 $c'$ with respect to a compactly supported isotopy in $\f{C}$ that fixes $D$
 pointwise. The following lemma about plane geometry of curves proves that
 $\iota(c,c')$ cannot be zero for $p>2$, which gives $HF^*(V_{c},V_{c'})\neq
 0$ contradicting the fact $V$ and $R(V)$ are disjoint and completes the proof
 of non-existence of exact Lagrangian submanifolds. \end{proof}

\begin{lemma} Let $p >q > 0$ be relatively prime integers, and $p=2r>2$ be an even
number. Let $D = \{ e^{2 \pi i k/p} : k = 0,1,\ldots, p-1 \}$.  Let $c : [0,1]
\to (\f{C}, D)$ be an oriented embedded curve such that $c(0)= e^{2 \pi i \kappa/ p }$
and $c(1) = - e^{2\pi i \kappa/p}$ for some $\kappa \in \{0,1,\ldots,p-1 \}$, and assume that the curve $-c(t)$ is isotopic to $c(-t)$ by a compactly supported isotopy in $\f{C}$ fixing $D$.

Further, let $c':[0,1] \to (\f{C}, D)$ be the curve given by $c'(t) = R(c(t)) = e^{2
\pi i q  /p } c(t)$.   Then the geometric intersection number of $c$ and $c'$ (the
minimal number of intersections among representatives of the isotopy classes
with respect to a compactly supported isotopy in $\f{C}$ fixing $D$) is non-zero.  \end{lemma}

\begin{proof}  For curves with ends on different points of $D$ we would like
to replace the geometric intersection number $\iota(\alpha, \beta)$ by an
algebraic one. Formally, we can consider $\f{C} \cup \{ \infty \}$ and take out small discs around the points of $D$
and , to get $\Sigma$ - a compact manifold with $p$
boundary circles $A_1, \ldots A_p$ (which we orient counterclockwise) on which the cyclic group
$\bfz_p$ still acts, with the generator $R$ sending $A_i$ to $A_{i+q}$ (as
usual $A_{p+k}=A_k$ for all $k$). We pair up the opposite boundary components
$B_i=A_i \cup A_{i+rq} $ (Note that $A_{i+rq} = R^r(A_i)$ is the circle that is diagonally opposite of $A_i$). Then, our curve $c$ represents a class in $H_1(\Sigma,
B_\kappa)$ and $c'$ represents a class in $H_1(\Sigma, B_{\kappa+1})$.
Lefschetz duality followed by the cup product gives a pairing: \[ \langle \ , \  \rangle \colon 
H_1(\Sigma, B_\mu) \times H_1(\Sigma, B_\nu) \to  H^2(\Sigma, \partial \Sigma)=\bfz \] Geometrically, for transverse curves $\alpha$ with
$[\alpha] =a$ and $\beta$ with $[\beta]=b$,  $\langle a,b \rangle$ is the number of
intersections of  $\alpha$ and $\beta$ counted with signs and in particular if
it is non-zero then $\iota(\alpha, \beta)$ is also non-zero.  We claim that
$\langle [c], [c'] \rangle \neq 0$, and the lemma follows from this. 
 
To compute $ \langle [c], [c'] \rangle$, write $[c]=l+b$ where $l$ is the class
represented by the linear path $l(t) = e^{2\pi i \kappa/p}(1-t) -e^{2\pi i \kappa/p}t$ connecting the endpoints of $c$. Then $b$ lies
in the image of $H_1(\Sigma)$ of the natural map $F: H_1(\Sigma) \mapsto
H_1(\Sigma, B_\kappa)$ in the homology exact sequence of the pair $(\Sigma,
B_\kappa)$, which is to say can be represented by union of closed curves in
$\Sigma$.  In fact, since $A_i$ for $i=1,\ldots, p$ form a basis of
$H_1(\Sigma)$,  we can write $b=F(\sum a_iA_i)$ and as the map $F$ above has
kernel spanned by $A_\kappa$ and $A_{\kappa+qr}$, there is a unique such
representation with $a_\kappa=a_{\kappa+qr}=0$.

Note that since $-c(t)$ is isotopic to $c(-t)$, in particular they are homologous,
so $R^r [c]=[-c] = -[c]$ (here by abuse of notation $R$ is used to denote the action
of the on $\Sigma$ sending  $H_1(\Sigma, B_\mu)$ to $ H_1(\Sigma, B_{\mu+1})$,
and the last minus sign comes form orientation reversal).

Combined with $R^r l= -l$, we have $R^r b= -b$, that is in the representation
$b=F(\sum a_iA_i)$ we must have $a_j=-a_{j+rq}$. Now, $[c']=R [c]= R l+ R b$ and we compute:
\begin{align*} \langle [c], [c'] \rangle &= \langle l, Rl \rangle + \langle l, Rb \rangle + \langle b, Rl \rangle + \langle b,Rb \rangle \\ &= 1 + (a_{\kappa+q}-a_{\kappa+q+rq}) + (a_{\kappa-q+rq}-a_{\kappa-q}) = 1 + 2 a_{\kappa+q} - 2 a_{\kappa-q} \end{align*} which is an odd integer, hence is non-zero; as desired.
\end{proof}

\begin{remark} For $p=2$, note that $S_{p-1}$ is exact symplectomorphic to $T^*S^2$ and $B_{p,q}$ is exact symplectomorphic to $T^* \f{R}P^2$ which indeed has its zero section as an exact Lagrangian submanifold. 
\end{remark} 

Having dealt with exact Lagrangian submanifolds, we next look for essential
Lagrangian tori. We observe that the tori $\f{T}_{p-1} \subset S_{p-1}$
considered in Section \ref{section2} are invariant under the action
$\Gamma_{p,q}$. We will henceforth be concerned with the Floer cohomology of
the quotient tori in $B_{p,q}$. We denote these tori by $\f{T}_{p,q }$. 

\begin{proposition}  $HF^*(\f{T}_{p,q }; \bfz_2)$ is
non-zero (and hence is isomorphic to $H^*(\f{T}_{p,q }; \bfz_2 )$ by
Proposition \ref{toriQH}).
\end{proposition} \begin{proof} Start with a pearl complex
$\mathcal{C}^*(\f{T}_{p,q };f,\rho,J)$ of $\f{T}_{p,q }$, given by some
generic Morse function $f$ and metric $\rho$ on $\f{T}_{p,q }$ and an
almost-complex structure $J$ on $B_{p,q}$. We can assume without loss
of generality that $f$ has a unique
maximum, giving rise to unique top degree generator $m$.   Consider the lifted
structures $f'$, $\rho'$ on $\f{T}_{p-1}$ and $J'$ on $S_{p-1}$.  Then since
every pearly trajectory in $S_{p-1}$ projects to one in $B_{p,q}$, and
conversely, every pearly trajectory in $B_{p,q}$ lifts to $S_{p-1}$ uniquely
given a starting point, if the triple  $(f, \rho, J)$ is regular for $(B_{p,q},
\f{T}_{p,q})$, the triple $(f', \rho', J')$ is regular for $(S_{p-1},
\f{T}_{p-1})$.

We see that $\mathcal{C}^*(\f{T}_{p-1};f',\rho',J')$ has $p$ top degree
generators $m_1, \ldots m_p$, the lifts of $m$, with $R m_i=m_{i+1}$ where $R$ is a generator of the cyclic group $\f{Z}_p$ acting on $S_{p-1}$ (we take $m_{p+1}=m_1$). We note that by equivariance $R d'm_i= d'
(Rm_i)$, and since the rank of the top degree homology is 1, the
element $M$ with $d'M=0$ must have $RM=M$ (since we are working over $\bfz_2$
coefficients). The only such element is $M=\sum m_i$. Then again, by
the
correspondence between the pearly trajectories,  $d' M$ is the (total) lift of
$dm$. Since this is 0, then so is $dm$, hence $m$ survives in cohomology. \end{proof}

\begin{proposition} If $T$ is a monotone Lagrangian 2-torus in a symplectic 4-manifold $X$, and $HF^* (T ; \bfz_2)\simeq H^* (T; \bfz_2)$, then $SH^*(X)$ is not zero. 

\end{proposition}
\begin{proof} This is essentially \cite[Proposition 5.2]{biased}. We only comment on the necessary modifications. Firstly, note that by Bockstein long exact sequence it suffices to show that $SH^*(X; \bfz_2 )\neq 0$ (cf. \cite[Remark 1.4]{recomb}).  

By using no auxiliary connection in all our Floer-theoretic constructions we
avoid the need to work over coefficient ring $\mathbb{K} \supset \bfq$, and use
$\bfz_2$ instead; additionally the fact that $\f{T}_{p,q }$ is monotone,
allows us to forego the Novikov ring coefficients and lift the requirement in \cite[Proposition 5.2]{biased} that
$T$ be Bohr-Sommerfeld.    Finally, being homologically essential over $\bfz_2 $ coefficients is by Proposition $\ref{toriQH}$ the same as $HF^*(\f{T}_{p,q} ;
\bfz_2) \neq0$. This allows one to repeat the arguments of Sections 5a and
5b of \cite{biased} to conclude $SH^* (X; \bfz_2)\neq 0$ just as in
\cite[Proposition 5.1]{biased}.\end{proof}

\begin{corollary} \label{nonempty} $SH^*(B_{p,q})$ is non-zero, in other words, $B_{p,q}$ is non-empty. \end{corollary} \QED

\begin{remark} Another way to prove $SH^*(B_{p,q})$ is non-zero goes as
follows: Since $S_{p-1}$ has exact Lagrangian submanifolds (matching spheres),
we conclude from Theorem $\ref{vit}$ that $SH^*(S_{p-1})$ is non-zero. Now,
there are obvious pull-back (total preimage) and push-forward (image) maps on
symplectic cohomology for unbranched covers which commute with the maps from
ordinary cohomology to symplectic cohomology, which shows that $SH^*(B_{p,q})
\neq 0$. Our method of proof above on the other hand yields a geometric reason
for the non-vanishing of $SH^*(B_{p,q})$.  
\end{remark}

\section{Concluding Remarks}

An exact Lefschetz fibration on $B_{p,q}$ can be found in \cite{endo}. This
Lefschetz fibration equips $B_{p,q}$ with a Stein structure, and Corollary
\ref{nonempty} implies that the symplectic cohomology is non-zero. It would be
interesting to use Seidel's computational methods (\cite{seidelD}) to compute
the symplectic cohomology of $B_{p,q}$ starting from this Lefschetz fibration.
Alternatively, $B_{p,q}$ can be constructed by Weinstein handle attachments with one 1-handle and one 2-handle to $D^4$ as in Section \ref{Leg}. The methods developed in \cite{BEE} might be useful in computing the symplectic cohomology from this description.

Let $D= \Pi_{p-1}^{-1} (0)$ be the fibre over the origin for $\Pi_{p-1} \colon S_{p-1} \to \f{C}$. There is a  special Lagrangian fibration on $S_{p-1} \backslash D$ with fibres $\f{T}_{r,\lambda} = \{ (x,y,z) \in S_{p-1} : |z|=r , |x|-|y| = \lambda \}$ (compare \cite[Section 5]{Aur}) where one could take the holomorphic volume form as in Lemma \ref{mas}. The matching tori that we considered in this paper corresponds to monotone fibres $\f{T}_{r,0}$ in this fibration. There is a unique singular fibre $\f{T}_{1,0}$ with $p$ nodal singularities.  In addition, this special Lagrangian fibration is equivariant under the action $\Gamma_{p,q}$ on $S_{p-1}$, hence it descends to a special Lagrangian fibration in the quotient $(S_{p-1} \backslash D) / \Gamma_{p,q}$ which has only one singular fibre with a unique nodal singularity. This construction gives an interesting testing ground for Strominger-Yau-Zaslow mirror symmetry conjecture and the related wall-crossing problem (cf. \cite{Aur}).

In this paper, we restricted our attention to dimension $4$. However, there is
a natural extension of our set-up to dimensions $4k$ for $k>1$. The action
$\Gamma_{p,q}$ exists and is free on the corresponding higher dimensional
$A_{p-1}$ Milnor fibre. We then obtain a non-displaceable Lagrangian $S^1
\times S^{2k-1}$ in the $A_{p-1}$ Milnor fibre and correspondingly, we get a
non-displaceable Lagrangian in the finite quotient.

\end{document}